\documentclass[10pt]{article}
\usepackage{amsmath}
\usepackage{amssymb}
\usepackage{amsthm}

\theoremstyle{plain} 
\newtheorem{thm}{Theorem}[section]

\newtheorem{prop}[thm]{Proposition}
\newtheorem{cor}[thm]{Corollary}
\theoremstyle{definition}
\newtheorem{defn}[thm]{Definition}
\theoremstyle{remark}
\newtheorem{rem}[thm]{Remark}
\newtheorem{ex}{Example}
\usepackage{mathrsfs}
\usepackage{bm}
\numberwithin{equation}{section}
\usepackage[pdftex]{graphicx}
\usepackage{here}
\usepackage{psfrag}

\title{Determining the optimal coefficient of the spatially periodic Fisher-KPP equation that minimizes the spreading speed}

\author{Ryo Ito
\\Meiji University, Tokyo, Japan }

\date{}
\begin{document}

\maketitle

\begin{abstract}
This paper is concerned with the spatially periodic Fisher-KPP equation $u_t=(d(x)u_x)_x+(r(x)-u)u$, $x\in \mathbb{R}$, where $d(x)$ and $r(x)$ are periodic functions with period $L>0$.
We assume that $r(x)$ has positive mean and $d(x)>0$.
It is known that there exists a positive number $c^*_d(r)$, called the minimal wave speed, such that a periodic traveling wave solution with average speed $c$ exists if and only if $c \geq c^*_d(r)$.
In the one-dimensional case, the minimal speed $c^*_d(r)$ coincides with the ``spreading speed'', that is, the asymptotic speed of the propagating front of a solution with compactly supported initial data.
In this paper, we study the minimizing problem for the minimal speed $c^*_d(r)$ by varying $r(x)$ under a certain constraint, while $d(x)$ arbitrarily.
We have been able to obtain an explicit form of the minimizing function $r(x)$.
Our result provides the first calculable example of the minimal speed for spatially periodic Fisher-KPP equations as far as the author knows.
\end{abstract}

\noindent
{\sc keywords}:{\ KPP equation; traveling wave; minimal speed; spreading speed} 






\section{Introduction}
Propagation phenomena appear in various fields of natural science, including population genetics, epidemiology, ecology and so on. 
The Fisher-KPP equation is among the classical models that describe propagation phenomena.
From the viewpoint of ecology, this equation describes the expansion of the territory of invading alien species in a given habitat.

In this paper, we investigate the spatially periodic Fisher-KPP equation:
\[
\tag*{$(E)$}u_t = (d(x)u_{x})_x + (r(x)-u)u, \ \ \ x \in \mathbb{R}, t>0,
\]
where $d(x)$ and $r(x)$ are periodic functions with period $L>0$.
While we always assume $d>0$, we will allow $r(x)$ to change sign: so long as its mean $\langle r \rangle_a$ is positive (see (\ref{asump})).


The solution $u(x,t)$ represents the population density of an invading species, while $r(x)$ denotes the intrinsic growth rate and $d(x)$ is the diffusion coefficient. 
These periodic coefficients represent an environment in which favorable zones and less favorable zones appear alternately in a periodic manner.

The Fisher-KPP equation was introduced by Fisher \cite{Fisher} and Kolmogorov, Petrovsky and Piskunov \cite{KPP} in 1937 in the content of population genetics. 
In 1951, Skellam \cite{S} used this equation as a model for biological invasion in ecology.
The above works were focused on the spatially homogeneous equation.
Shigesada, Kawasaki and Teramoto \cite{S1} in 1986 considered the case where the coefficients are spatially periodic and studied the influence of periodic environments on the invasion speed.
The paper \cite{S1} introduced a notion of traveling wave solutions in the periodic setting, while they called ``traveling periodic solution".

Berestycki and Hamel \cite{BH} proved the existence of periodic traveling waves for the positive coefficient $r(x)$.
They also proved that the slowest traveling wave exists.
We call its speed the minimal traveling wave speed (or minimal speed in short) and it is denoted by $c_d^*(r)$, that is, the traveling wave with average speed $c$ exists if and only if $c \geq c^*_d(r)$. 

Weinberger \cite{W} also studied periodic traveling waves together with the ``spreading speed'' in a rather abstract setting that include reaction-diffusion equation of the form (E) as a special case.
The term ``spreading speed'' refers to the asymptotic speed of the propagating front of a solution with compactly supported initial data.
Under the assumption that $u \equiv 0$ is unstable the existence of the spreading speed and periodic traveling waves was proved.
He also derived that the spreading speed coincides with the minimal speed in the one-dimensional case and the spreading speed is characterized by using the principal eigenvalue of the corresponding linearized operator.


Berestycki-Hamel-Nadirashvili \cite{BHN1} proved that the minimal speed $c_d^*(r)$ is characterized by the following formula:
\[
c^*_d(r)=\min_{\lambda>0} \Big( -\cfrac{k_{\lambda}(d,r)}{\lambda}\Big),
\]
where $k_{\lambda}(d,r)$ is the principal eigenvalue of a certain operator $-\mathcal{L}_{\lambda,d,r}$.
The variational characterization of the principal eigenvalue $k_{\lambda}(d,r)$ has been derived by Nadin \cite{N1}. 
Hence we can analyze the minimal speed by using variational method.
See also subsection \ref{problem}.

The purpose of this work is to analyze the influence of periodic environment on the invasion speed.
Specifically, in this paper, we consider the problem of finding a minimizing function of $c_d^*(r)$ varying $r(x)$, where $d \in C^{1+\delta}_{\mathrm{per}}(\mathbb{R})$ is fixed and a minimizer $r(x)$ is sought in
\[
\varLambda(\alpha) := \{ \ r \in C^{\delta}_{\mathrm{per}}(\mathbb{R}) \mid \frac{1}{L} \int_{0}^{L} r(x) dx = \alpha \  \}.
\]
Here $\delta>0$ is given positive constant.
In other words, we consider the following minimizing problem.
\[
\tag*{$(P)_d$} \underset{r \in \varLambda(\alpha)}{\mathrm{Minimize}} \  c^*_d(r)
\]
From the ecological point of view, the spreading speed describes the invasion speed of alien species.
Hence the problem means seeking the best disposition of environment to prevent the invasion of alien species.

A minimizing problem associated with the minimal speed is partially discussed in Shigesada-Kawasaki-Teramoto \cite{S1}.
They studied the dependence of the period $L>0$ to $c^*_{d}(r)$ under the certain  assumption, and they proved that $L \mapsto c^*_{d}(r)$ is nondecreasing.
Their work was partly unrigorous from mathematical point of view because their analysis was based on a formal asymptotic representation of traveling waves. 
Nadin \cite{N1} gave the rigorous proof of this research by dealing with much more general equations.

In the case where $d(x)$ is a constant, 
Berestycki-Hamel-Roques \cite{BHR2} derived that a constant function minimizes the minimal speed, and Liang-Lin-Matano \cite{M1} proved that the principal eigenfunction is a constant function if $r(x)$ is constant.
These results are derived by using the eigenvalue problem associated with the operator $-\mathcal{L}_{\lambda,d,r}$.

These previous works suggest that the minimal speed will be slower if  environments are more homogenized, 
and the most averaged environment minimizes the spreading speed.

In the case of sinusoidal diffusion and growth coefficient, N. Kinezaki, K. Kawasaki and N. Shigesada \cite{S3} computed the minimal speed varying the phase of the diffusion coefficient.
By numerically solving the equation, they concluded that the minimal speed attains its minimum (maximum) when the diffusion and the growth coefficient have same (opposite) phases.
Nadin \cite{N1} formulated this numerical result about maximizing the speed as a conjecture using the Schwarz rearrangement, and he studied the influence of the concentrating effect on the minimal speed when the diffusion coefficient is not constant.
A maximizing problem is also investigated by some researchers.
See also \cite{N1,M1,M2,Xiao,I}.
The effect of temporal averaging on the minimal speed is also considered by Nadin \cite{N2}.

In this paper, we consider the case where $d(x)$ is a periodic function.
The main difficulty with this problem is that the eigenvalue problem is more complicated than the constant case.
As we will see later, in the periodic case, the principal eigenfunction is not a constant function.
See subsection \ref{Main}.
The mathematical motivation of this work is to analyze how the optimal growth coefficient depends on the fixed diffusion coefficient.

In 2010, Nadin derived the following inequality
\begin{equation}\label{ine}
c^*_d(r) \geq 2\sqrt{\langle\, d\, \rangle_h\langle\, r\, \rangle_a}
\end{equation}
Here $\langle\, r\, \rangle_a$ is the spatial arithmetic mean of $r(x)$, and $\langle\, d\, \rangle_h$ is the spatial harmonic mean of $d(x)$, that is, the symbols $\langle\, r\, \rangle_a,\, \langle\, d\, \rangle_h$ are defined by
\begin{equation}\label{mean}
\langle\, r\, \rangle_a = \cfrac{1}{L} \int_{0}^{L} r(x) \ dx,\ \ \langle\, d\, \rangle_h =\cfrac{1}{\displaystyle{\cfrac{1}{L}\int_{0}^{L} \cfrac{1}{d(x)}\  dx}}
\end{equation}
for any $r \in C_{\mathrm{per}}(\mathbb{R})$ and $d \in C_{\mathrm{per}}^1(\mathbb{R})$.
We solve the minimizing problem $(P)_d$ by finding out a condition under which equality holds in the inequality (\ref{ine}).

We will see equality in (\ref{ine}) holds if and only if $d(x)$ and $r(x)$ satisfy the following relational expression.
\begin{equation}\label{condition}
\cfrac{r}{\langle\, r\, \rangle_a} + \cfrac{\langle\, d\, \rangle_h}{d} = 2
\end{equation}

By the condition (\ref{condition}), we see that the minimizing problem $(P)_d$ has the solution for any $d \in C^{1+\delta}_{\mathrm{per}}(\mathbb{R})$ with $\inf d >0$, that is,
\begin{equation}\label{minimizer}
r_d(x) = \alpha \Big(2-\cfrac{\langle\, d\, \rangle_h}{d(x)}\Big),\ \ \ \ x \in \mathbb{R}
\end{equation}
is the solution for the minimizing problem $(P)_d$.
The condition (\ref{condition}) is first introduced by El Smaily-Hamel-Roques \cite{condition} in a study on an approximate value of the spreading speed. 
See also subsection \ref{problem}.
In this paper, we will rediscover this condition to find the optimal coefficient.

By (\ref{minimizer}), the spreading speed attains its minimum when $r(x)$ is large in the area where $d(x)$ is large and $r_d(x)$ is small in the area where $d(x)$ is small.
See also subsection \ref{Main} and subsection 3.1.

The interpretation of our main result from the ecological point of view is that the invasion speed of alien species reaches its minimum when the species quickly disperse in their favorable areas and slowly disperse in their less favorable areas.

Our result provides the influence of a non-trivial relation between the shape of the diffusion coefficient and the growth coefficient on the spreading speed.
In some sense, our result formulated a numerical result computed by Kinezaki-Kawasaki-Shigesada \cite{S3} in a different way from Nadin.

By the effect of the surrounding environment, the most averaged function is not the minimizing function.
As far as the minimizing problem associated with the spreading speed (or the minimal traveling wave speed), this work provides the first example finding out the influence of the shape of the heterogeneity of the diffusion coefficient on the optimal growth coefficient.

Our result means that $c_d^*(r)=2\sqrt{\langle\, d\, \rangle_h\langle\, r\, \rangle_a}$ when $(d,r)$ satisfy the condition $(\ref{condition})$.
An approximate value of the spreading speed is known when $L \to 0$ and $L \to \infty$, but the exact value is only known in the case where $d$ and $r$ are constant as far as the author knows.
This is the first calculable example of the spreading speed (or the minimal  traveling wave speed) for the spatially periodic Fisher-KPP equation.

This paper is organized as follows:
In section 2, we state our main results, and we introduce known results of the spreading speed.
In section 3, we give the proofs of our main results.

	
	
	
	
	
	
	

\section{Main results}
In this section, we explain the minimal speed of traveling waves and state the main results as well as the explanation of related works.
\subsection{Formulation of the problem}\label{problem}
In this subsection, we recall some known results of the spatially periodic Fisher-KPP equation.
We consider the following Cauchy problem:
\[
\tag*{$(E_0)$}
\begin{cases}
u_t = (d(x)u_{x})_x + (r(x)-u)u, \ \ \ x \in \mathbb{R}, t>0,\\
u(x,0) = u_0(x) \geq 0, \ \ \ x \in \mathbb{R},
\end{cases}
\]
where $u_0 \in C_c(\mathbb{R}),\,u_0 \geq 0,\, u_0 \not \equiv 0$.
In what follows, we assume that
\begin{equation}\label{asump}
\inf d >0,\, \langle \, r\, \rangle_a >0.
\end{equation}
In this case, a stationary problem of $(E)$ has the positive periodic solution $p(x)$, that is, there exists the positive function $p(x)$ that satisfies
\[
(d(x)p_{x})_x + (r(x)-p)p=0.
\]
See Berestycki-Hamel-Roques \cite{BHR1}.

Weinberger \cite{W} and Berestycki-Hamel-Roques \cite{BHR1,BHR2} proved that (E) has traveling wave solutions.
\begin{defn}[Periodic traveling waves]
	A solution $u(x,t):\mathbb{R} \times \mathbb{R} \to \mathbb{R}$ of $(E)$ is called a periodic traveling wave solution in the positive direction if the following conditions hold:
	\begin{itemize}
		\item[\rm(1)]  $\underset{x \to  -\infty}{\lim}(u(x,t)-p(x)) = 0$, $\underset{x \to \infty}{\lim}u(x,t) = 0$ locally uniformly in $t \in \mathbb{R}$;
		\item[\rm(2)] There exists a constant $T>0$ such that
		\[
		u(x-L,t) = u(x,t+T) \ \ \ (x,t) \in \mathbb{R} \times \mathbb{R}.
		\]
	\end{itemize}
\end{defn}
Here we call the quantity $c := L/T$ the average speed of the traveling wave $u(x,t)$ (or ``speed" for simplicity).

They also proved that there exists the minimal traveling wave speed $c^*_d(r)$ (or ``minimal speed'' for simplicity), that is, the traveling wave with speed $c$ exists if and only if $c \geq c^*_d(r)$. 


The Cauchy problem $(E_0)$ has the classical global solution $u(x,t)$ for any $u_0 \in C_c(\mathbb{R}),\,u_0 \geq 0,\, u_0 \not \equiv 0$.
Furthermore, trivial solution $0$ is unstable under the assumption (\ref{asump}).
It means that the solution $u(x,t)$ goes to the positive function $p(x)$ as $t \to \infty$.
The speed of an expanding front of $u(x,t)$ asymptotically approaches to a certain value as $t \to \infty$.
\begin{defn}[Spreading speed]
	A quantity $\tilde{c}>0$ is called the spreading speed if for any nonnegative initial data $u_0 \in C_c(\mathbb{R})$ with $u_0 \geq 0,\, u_0 \not \equiv 0$, the solution $u(x,t)$ of the Cauchy problem with initial data $u_0$ satisfies that
	\begin{itemize}
		\item[\rm(1)] $\underset{t \to \infty}{\lim} u(ct,t) = 0$ if $c > \tilde{c}$,
		\item[\rm(2)] $\underset{t \to \infty}{\liminf} u(ct,t) > 0$ if $0 < c < \tilde{c}$.
	\end{itemize}
\end{defn}

Weinberger \cite{W} and Berestycki-Hamel-Nadirashvili \cite{BHN1} proved that the minimal speed $c^*_d(r)$ is the spreading speed in the one-dimensional case.


Any traveling wave solution $u(x,t)$ with speed $c>c^*_d(r)$ in the negative direction has the following asymptotic expression if $(x,t)$ satisfies $u(x,t) \approx 0$:
\begin{equation}\label{eq2.5.9}
u(x,t) \sim e^{\lambda (x+ct)} \psi(x),
\end{equation}
where $\psi>0$ is some $L$-periodic function and $\lambda>0$ is some constant.

In $u(x,t) \approx 0$, $(r(x)-u)$ is practically equal to the intrinsic growth rate $r(x)$.
Substituting (\ref{eq2.5.9}) into the equation $(E)$, we have
\begin{equation*}
-(d(x) \psi'(x))' - 2\lambda d(x) \psi'(x) -(\lambda^2 d(x) + \lambda d'(x) +r(x)) \psi(x) = - \lambda c \psi(x), \ \ \ x \in \mathbb{R}
\end{equation*}

Set the operator $-\mathcal{L}_{\lambda,d,r}$ on $C^2_{\mathrm{per}}(\mathbb{R})$ for any constant $\lambda>0$ as follows
\begin{equation*}
-\mathcal{L}_{\lambda,d,r} \psi(x) := - (d(x) \psi'(x))' - 2\lambda d(x) \psi'(x) -(\lambda^2 d(x) + \lambda d'(x) +r(x)) \psi(x).
\end{equation*}
and we denote by $k_{\lambda}(d,r)$ the principal eigenvalue of the operator $-\mathcal{L}_{\lambda,d,r}$, that is, 
\begin{equation}\label{eproblem}
\begin{cases}
-\mathcal{L}_{\lambda,d,r} \psi = k_{\lambda}(d,r)\psi,\\
\psi(x+L)\equiv \psi(x),
\end{cases}
\end{equation}
and the eigenfunction $\psi$ is positive.

It is expected by the above formal calculation that $-c\lambda$ is the principal eigenvalue of the operator $-\mathcal{L}_{\lambda,d,r}$, that is, 
\[
-c\lambda = k_{\lambda}(d,r).
\]
Recall that $c^*_d(r)$ is the minimal speed, we expect
\begin{equation}\label{ck}
c^*_d(r) = \min_{\lambda>0}\Big( - \cfrac{k_{\lambda}(d,r)}{\lambda}\Big),
\end{equation}
and this formula was established by Berestycki-Hamel-Nadirashvili \cite{BHN1}.

In the case where $d(x)$ is a constant, Liang-Lin-Matano \cite{M1} proved that equality in the inequality (\ref{ine}) holds if and only if $r(x)$ is a constant.
They also derived that the principal eigenfunction of the operator $-\mathcal{L}_{\lambda_0,d,r}$ is constant, where $\lambda_0 = \sqrt{\langle\, r\, \rangle_a/d}$ satisfies
\[
c^*_d(r) = - \cfrac{k_{\lambda_0}(d,r)}{\lambda_0}. 
\]
These results are proved by using the eigenvalue problem (\ref{eproblem}).

However, if $d(x)$ is not constant, the eigenvalue problem (\ref{eproblem}) is more complicated than the constant case.
As we will see later (Theorem \ref{main4}), the principal eigenfunction is not constant.
That is why we use the formula about the principal eigenvalue derived by Nadin which is simpler than (\ref{eproblem}).

The first eigenvalue of $-\varDelta$ is represented by a integral functional (the Rayleigh characterization).
Nadin \cite{N1} gives the following representation of $k_{\lambda}(d,r)$ that is the principal eigenvalue of the non-symmetric operator $-\mathcal{L}_{\lambda,d,r}$.
\begin{prop}[Nadin \cite{N1}]\label{tr3}
	Set $E_L := \{ \ \varphi \in C^1_{\mathrm{per}} \mid \varphi > 0, \ \int_{0}^{L} \varphi^2 dx =1 \  \}$.
	The principal eigenvalue $k_{\lambda}(d,r)$ of $-\mathcal{L}_{\lambda,d,r}$ is characterized as follows:
	\[
	k_{\lambda}(d,r) = \min_{\varphi \in E_L} \Biggr\{ \int_{0}^{L}d|\varphi '|^2 \ dx -\int_{0}^{L} r \varphi^2 \ dx - \cfrac{\lambda^2L^2}{\displaystyle{\int_{0}^{L} \cfrac{1}{d\varphi^2}\  dx}}  \Biggl\}.
	\]
\end{prop}

By using this formula, Nadin studied the dependence of the period $L>0$ on the minimal speed $c_d^*(r)$, and derived the lower estimate of the minimal speed which we investigate as a corollary.
Define
\[
d_L(x)=d(x/L),\, r_L(x)=r(x/L),
\]
where $d(x)$ and $r(x)$ are 1-periodic functions.
Set $c^*_L=c_{d_L}^*(r_L)$.
\begin{prop}[Nadin \cite{N1}]\label{tr4}
	The following statements hold:
	\begin{itemize}
		\item[$(1)$] The function $L \mapsto k_{\lambda}(d_L, r_L)$ and $L \mapsto c^*_L$ are nondecreasing.
		Moreover,
		\begin{eqnarray*}
			\lim_{L\to 0} k_{\lambda}(d_L, r_L) &=& -\langle\, r\, \rangle_a -\lambda^2 \langle\, d\, \rangle_h,\\
			\lim_{L\to 0} c^*_L &=& 2\sqrt{\langle\, d\, \rangle_h\langle\, r\, \rangle_a}.
		\end{eqnarray*}
		\item[$(2)$] For any $L>0$,
		\[
		c^*_L \geq 2\sqrt{\langle\, d\, \rangle_h\langle\, r\, \rangle_a}.
		\]
	\end{itemize}
\end{prop}
The condition (\ref{condition}) is first introduced by El Smaily-Hamel-Roques.
\begin{prop}[El Smaily-Hamel-Roques \cite{condition}]\label{LSm}
	For some $L_0>0$, the map $L \mapsto c^*_L$ is in $C^{\infty}(0,L_0)$.
	Moreover,
	\[
	\lim_{L\to 0} \cfrac{dc^*_L}{dL}=0,\,\lim_{L\to 0} \cfrac{d^2c^*_L}{dL^2}\geq 0.
	\]
	Finally, the following two statements are equivalent:
	\begin{itemize}
		\item[$(1)$] $\underset{L\to 0}{\lim}\cfrac{d^2c^*_L}{dL^2}> 0$.
		\item[$(2)$] $\cfrac{r}{\langle\, r\, \rangle_a} + \cfrac{\langle\, d\, \rangle_h}{d} \neq 2$.
	\end{itemize}
\end{prop}

\subsection{Main results}\label{Main}
For any $\varphi \in E_L,\, \lambda>0,\,r \in C_{\mathrm{per}}(\mathbb{R}),\, d \in C^1_{\mathrm{per}}(\mathbb{R})$, set
\[
I(\varphi;\lambda,d,r):=\int_{0}^{L}d|\varphi '|^2 \ dx -\int_{0}^{L} r \varphi^2 \ dx - \cfrac{\lambda^2L^2}{\displaystyle{\int_{0}^{L} \cfrac{1}{d\varphi^2}\  dx}}
\]
By Proposition \ref{tr4}, we can rewrite the principal eigenvalue $k_{\lambda}(d,r)$ as
\[
k_{\lambda}(d,r) =\min_{\varphi \in E_L} I(\varphi;\lambda,d,r).
\]
Hence the spreading speed $c^*_d(r)$ is also rewrote as
\begin{equation}\label{minispeed}
c^*_d(r) = -\max_{\lambda>0} \min_{\varphi \in E_L} \lambda^{-1} I(\varphi;\lambda,d,r)
\end{equation}
from the formula (\ref{ck}).
Now we state our main results.
\begin{thm}\label{main3}
	The following statements are equivalent:
	\begin{itemize}
		\item[$(1)$] $c^{*}_{d}(r) = 2 \sqrt{\langle\, d\, \rangle_h \langle\, r\, \rangle_a}$.
		\item[$(2)$] $\cfrac{r}{\langle\, r\, \rangle_a} + \cfrac{\langle\, d\, \rangle_h}{d} = 2$.
		\item[$(3)$] Set $\lambda_0=\sqrt{\langle\, r\, \rangle_a/\langle\, d\, \rangle_h}$ and $\varphi_0 \equiv 1/\sqrt{L}$.
		Then
		\begin{equation}\label{var}
		c^*_d(r) = -\lambda_0^{-1} I(\varphi_0;\lambda_0,d,r).
		\end{equation}
	\end{itemize}
	Moreover, $(\lambda_0,\varphi_0)$ is unique pair that satisfy (\ref{var}). 
\end{thm}

\begin{rem}
	This theorem is generalization of the result proved by Berestycki-Hamel-Roques \cite{BHR2} in some sense.
	In the case where $d$ is a constant, we can easily see that $\langle\, d\, \rangle_h =d$.
	Thus the condition (\ref{condition}) is rewrote as follows:
	\[
	r(x) = \langle\, r\, \rangle_a 
	\]
	for any $x \in \mathbb{R}$.
	It means that equality in (\ref{ine}) holds if and only if $r$ is a constant.
\end{rem}

Theorem \ref{main3} implies that the problem $(P)_d$ has the solution.
Furthermore, we see that the solution does not depend on size of the diffusion coefficient.
It depends on only the shape of the diffusion coefficient. 
\begin{cor}\label{main2}
	For any $d \in C^{1+\delta}_{\mathrm{per}}(\mathbb{R})$ with $\inf d > 0$, the minimizing problem $(P)_d$ has the solution $r_d(x)$ and it is defined by
	\[
	r_d(x) = \alpha \Big(2-\frac{\langle\, d\, \rangle_h}{d(x)}\Big)
	\]
	for any $x \in \mathbb{R}$.
	Moreover, for any $k>0$, $r_d$ is the unique solution for $(P)_{kd}$.
\end{cor}

\begin{rem}
	The right-hand side of (\ref{minispeed}) can be defined even if $d \in C^{1}_{\mathrm{per}}(\mathbb{R})$ with $\inf d > 0$ and $r \in C_{\mathrm{per}}(\mathbb{R})$ with $\langle\, r\, \rangle_a>0$.
	For any $d \in C^{1}_{\mathrm{per}}(\mathbb{R})$ and $r \in C_{\mathrm{per}}(\mathbb{R})$, our results still hold if we formally define the spreading speed by the right-hand side of (\ref{ck}).
\end{rem}

Finally, we state that the principal eigenfunction is not constant in the periodic case.
\begin{thm}\label{main4}
	Assume that (\ref{condition}) holds and $\lambda_0=\sqrt{\langle\, r\, \rangle_a / \langle\, d\, \rangle_h }$.
	Set $\psi_c \equiv C$ for some constant $C\neq 0$.
	Then the following statements are equivalent:
	\begin{itemize}
		\item[$(1)$] $\psi_c$ is a principal eigenfunction of the operator $-\mathcal{L}_{\lambda_0,d,r}$.
		\item[$(2)$] The diffusion coefficient $d(x)$ is a constant function.
	\end{itemize}
\end{thm}

\section{Examples and proof}
\subsection{Specific examples of the main results}
In this subsection, we give some examples of our main result.
We notice that the exact value of the minimal traveling wave speed is calculable if the condition (\ref{condition}) establish.
One example is as follows.
This is the first calculable example of the minimal speed for spatially periodic Fisher-KPP equations.

\begin{ex}
	Define $r(x)$ and $d(x)$ by
	\[
	r(x) =1+\frac{1}{2} \sin x, \  d(x) = \cfrac{1}{1-\frac{1}{2} \sin x}
	\]
	for any $x \in \mathbb{R}$.
	Then, $r$ and $d$ satisfy the condition (\ref{condition}).
	By Theorem \ref{main3}, we can calculate the minimal traveling wave speed $c_d^*(r)$ for the equation $(E)$ as follows:
	\[
	c_d^*(r) =2\sqrt{\langle\, d\, \rangle_h\langle\, r\, \rangle_a} = 2.
	\]
	Then we obtain the exact value of the minimal speed.
\end{ex}

\begin{rem}
	We notice that $r_d(x)$ is large in the area where $d(x)$ is large and $r_d(x)$ is small in the area where $d(x)$ is small.
	See (\ref{minimizer}). 
\end{rem}







If $r(x)$ is fixed, equality in (\ref{ine}) may not hold by varying $d(x)$. 

\begin{ex}
	Define $r(x)$ by
	\[
	r(x) =1+2 \sin x
	\]
	for any $x \in \mathbb{R}$.
	Then
	\[
	d_r^{-1}(x)= 2 -r(x) =1-2\sin x
	\]
	and $r$ satisfy the condition (\ref{condition}). 
	However, $d_r(x)<0$ on $(\pi/3,2\pi/3)$.
	We conclude that
	\[
	c_d^*(r) > 2\sqrt{\langle\, d\, \rangle_h\langle\, r\, \rangle_a}
	\]
	for any $d \in C^1_{\mathrm{per}}(\mathbb{R})$ with $\inf d >0$.
\end{ex}



\subsection{Proof of the lower estimate}
For the readers' convenience, we give the proof of the inequality (\ref{ine}) in this subsection.
This inequality was first proved by Nadin \cite{N1}. 
\begin{proof}[Proof of (\ref{ine})]
	By Nadin's formula, we have
	\[
	k_{\lambda}(d,r) = \min_{\varphi \in E_L} \Biggr\{ \int_{0}^{L}d|\varphi '|^2 \ dx -\int_{0}^{L} r \varphi^2 \ dx - \cfrac{\lambda^2L^2}{\displaystyle{\int_{0}^{L}} \cfrac{1}{d\varphi^2}\  dx}  \Biggl\}= \min_{\varphi \in E_L} I(\varphi;\lambda,d,r).
	\]
	Taking $\varphi_0(x)=1/\sqrt{L}$ as a test function, we obtain
	\begin{eqnarray*}
	k_{\lambda}(d,r) &\leq&  \int_{0}^{L}d|\varphi '|^2 \ dx -\int_{0}^{L} r \varphi^2 \ dx- \cfrac{\lambda^2L^2}{\displaystyle{\int_{0}^{L}} \cfrac{1}{d\varphi_0^2}\ dx }\\
	&=& -\langle\, r\, \rangle_a -\lambda^2 \langle\, d\, \rangle_h.
	\end{eqnarray*}
	Therefore
	\begin{eqnarray*}
		c^*_d(r) &=& \min_{\lambda >0} -\cfrac{k_{\lambda}(d,r)}{\lambda}\\
		& \geq & \min_{\lambda >0} \Bigr( \cfrac{\langle\, r\, \rangle_a}{\lambda} +\lambda \langle\, d\, \rangle_h \Bigl)\\
		&=& 2\sqrt{\langle\, d\, \rangle_h \langle\, r\, \rangle_a}.
	\end{eqnarray*}
	The inequality $(\ref{ine})$ is proved.
\end{proof}

\subsection{Proof of the main results}
\begin{proof}[Proof of Theorem \ref{main3}]
	We first prove that $(1) \Rightarrow (2)$.
	Assume that the equality holds in (\ref{ine}).
	Define $\lambda^*>0$ and $\varphi_{\lambda}^* \in E_L$ by
	\[
	c^*_d(r) =-\cfrac{k_{\lambda^*}(d,r)}{\lambda^*},\  k_{\lambda}(d,r)= I(\varphi_{\lambda}^*;\lambda,d,r)
	\]
	for any $\lambda>0$.
	As in the proof of (\ref{ine}), for any $\lambda>0$, we have
	\[
	k_{\lambda}(d,r)= I(\varphi^*;\lambda,d,r) \leq I(\varphi_0;\lambda,d,r)=-\langle\, r\, \rangle_a -\lambda^2 \langle\, d\, \rangle_h.
	\]
	Therefore
	\[
	c^*_d(r) =-\cfrac{k_{\lambda^*}(d,r)}{\lambda^*} \geq \cfrac{\langle\, r\, \rangle_a}{\lambda^*} +\lambda^* \langle\, d\, \rangle_h \geq 2\sqrt{\langle\, d\, \rangle_h \langle\, r\, \rangle_a}.
	\]
	From the assumption $c^*_d(r) =2\sqrt{\langle\, d\, \rangle_h \langle\, r\, \rangle_a}$, it follows that
	\begin{equation}\label{minmax}
	\cfrac{\langle\, r\, \rangle_a}{\lambda^*} +\lambda^* \langle\, d\, \rangle_h = 2\sqrt{\langle\, d\, \rangle_h \langle\, r\, \rangle_a},\  k_{\lambda^*}(d,r)= I(\varphi_0;\lambda^*,d,r).
	\end{equation}
	This implies that $\lambda^*=\lambda_0=\sqrt{\langle\, r\, \rangle_a/\langle\, d\, \rangle_h}$ and the constant function $\varphi_0$ minimizes the functional $\varphi \mapsto I(\varphi;\lambda^*,d,r)$ on $E_L$.
	The constant function $\varphi_0$ also minimizes the following functional
	\[
	I(\varphi) =\cfrac{1}{\displaystyle{\int_{0}^{L}} \varphi^2 dx}\Biggr\{ \int_{0}^{L}d|\varphi '|^2dx - \int_{0}^{L} r \varphi^2 dx - \cfrac{\lambda_0^2L^2}{\displaystyle{\int_{0}^{L}} \cfrac{1}{d\varphi^2}\ dx }   \Biggl\}
	\]
	for $\varphi \in C^1_{\mathrm{per}} \setminus \{0\}$.
	We next calculate the Euler-Lagrange equation for the functional $I$ on $C^1_{\mathrm{per}} \setminus \{0\}$.
	Take a minimizing function $\varphi \in C^1_{\mathrm{per}} \setminus \{0\}$ with $\|\varphi \|_{L^2} =1$.
	For any $L$-periodic function $\psi \in C^1(\mathbb{R})$ and sufficiently small $\varepsilon >0$, we see that $\varphi+\varepsilon \psi \in C^1_{\mathrm{per}} \setminus \{0\}$ and
	\[
	I(\varphi;\lambda_0,d,r)=I(\varphi) \leq I(\varphi+\varepsilon \psi).
	\]
	Since 
	\[
	\cfrac{d}{d\varepsilon} I(\varphi+\varepsilon \psi)\Big|_{\varepsilon=0} =0
	\]
	and $\|\varphi \|_{L^2}=1$, we obtain
	\[
	\int_{0}^{L}d \varphi' \psi' dx - \int_{0}^{L}\varphi \psi\, dx -\cfrac{\lambda^2_{0}L^2}{\Big(\displaystyle{\int_{0}^{L}}\cfrac{1}{d \varphi^2}\,dx\Big)^2}\int_{0}^{L}\cfrac{\psi}{d \varphi^3}\,dx=I(\varphi) \int_{0}^{L} \varphi \psi \,dx.
	\]
	It follows that the minimizing function $\varphi$ satisfies the following Euler-Lagrange equation in the weak sense:
	\[
	- (d\varphi')' -r\varphi -\cfrac{\lambda^2_{0}L^2}{\Big(\displaystyle{\int_{0}^{L}}\cfrac{1}{d \varphi^2}\,dx\Big)^2} \cfrac{1}{d\varphi^3}=I(\varphi) \varphi.
	\]
	Substituting $\varphi_0 \equiv 1/\sqrt{L}$ into the Euler-Lagrange equation we have
	\begin{equation}\label{EL1}
	-\cfrac{1}{\sqrt{L}}\  r - \cfrac{\lambda^2_0 L^{2/3}}{d\Big(\displaystyle{\int_{0}^{L}} \cfrac{1}{d}\, dx\Big)^2} =\cfrac{1}{\sqrt{L}}\  I(\varphi_0). 
	\end{equation}
	Since
	\begin{equation}\label{EL2}
	\langle\, d\, \rangle_h=\Big(\frac{1}{L}\int_{0}^{L} \frac{1}{d(x)}\ dx\Big)^{-1},\  \lambda_0 = \sqrt{\langle\, r\, \rangle_a/\langle\, d\, \rangle_h},
	\end{equation}
	we obtain 
	\begin{equation}\label{EL3}
	I(\varphi_0)=-\langle\, r\, \rangle_a - \lambda_0^2 \langle\, d\, \rangle_h =-2\langle\, r\, \rangle_a.
	\end{equation}
	From (\ref{EL2}) and (\ref{EL3}), we can rewrite (\ref{EL1}) as
	\[
	\cfrac{r}{\langle\, r\, \rangle_a} + \cfrac{\langle\, d\, \rangle_h}{d} = 2,
	\]
	which is the desired conclusion.
	
	We next prove that $(1) \Leftarrow (2)$.
	Since we know $c^*_d(r_d)\geq 2\sqrt{\langle\, d\, \rangle_h\langle\, r_d\, \rangle_a}$, it is sufficient to prove that the converse inequality.
	We have
	\begin{equation}\label{ine5}
	c^*_d(r_d)=\min_{\lambda>0} \Big(- \cfrac{k_{\lambda}(d, r_d)}{\lambda}\Big) \leq - \cfrac{k_{\lambda_0}(d, r_d)}{\lambda_0},
	\end{equation}
	where $\lambda_0=\sqrt{\langle\, r_d\, \rangle_a/\langle\, d\, \rangle_h}$.
	Nadin's formula and (\ref{minimizer}) show that
	\begin{eqnarray*}
	k_{\lambda_0}(d, r_d) &=& \min_{\varphi \in E_L} \Biggr\{  \int_{0}^{L}d|\varphi '|^2 \ dx -\int_{0}^{L} r_d \varphi^2 \ dx - \cfrac{\lambda_0^2L^2}{\displaystyle{\int_{0}^{L} \cfrac{1}{d\varphi^2}\  dx}} \Biggl\}\\
	&\geq&  \min_{\varphi \in E_L} \Biggr\{ -\int_{0}^{L} r_d \varphi^2 \ dx - \cfrac{\lambda_0^2L^2}{\displaystyle{\int_{0}^{L} \cfrac{1}{d\varphi^2}\  dx}} \Biggl\}\\
	&=& -2\langle\, r_d\, \rangle_a+ \min_{\varphi \in E_L} \Biggr\{ \langle\, d\, \rangle_h\langle\, r_d\, \rangle_a \int_{0}^{L} \cfrac{1}{d} \ \varphi^2 \ dx - \cfrac{\lambda_0^2L^2}{\displaystyle{\int_{0}^{L} \cfrac{1}{d\varphi^2}\  dx}} \Biggl\}.
	\end{eqnarray*}
	By the Cauchy-Schwarz inequality, we have
	\begin{equation*}
	\Big(\int_{0}^{L} \cfrac{1}{d} \ \varphi^2 \ dx\Big) \Big(\int_{0}^{L} \cfrac{1}{d\varphi^2}\ dx\Big) \geq \Big(\int_{0}^{L} \cfrac{1}{d}\ dx\Big)^2=L^2 \langle\, d\, \rangle_h^{-2}.
	\end{equation*}
	We thus get
	\begin{equation*}
	\langle\, d\, \rangle_h\langle\, r_d\, \rangle_a \int_{0}^{L} \cfrac{1}{d} \ \varphi^2 \ dx - \cfrac{\lambda_0^2L^2}{\displaystyle{\int_{0}^{L} \cfrac{1}{d\varphi^2}\  dx}}   \geq
	\cfrac{L^2 (\langle\, d\, \rangle_h^{-1}\langle\, r\, \rangle_a -\lambda_0^2)}{\displaystyle{\int_{0}^{L} \cfrac{1}{d\varphi^2}\  dx}} =0
	\end{equation*}
	for any $\varphi \in E_L$.
	This gives
	\begin{equation}\label{ine4}
	k_{\lambda_0}(d, r_d)\geq -2\langle\, r_d \, \rangle_a.
	\end{equation}
	Combining (\ref{ine4}) with (\ref{ine5}), we obtain
	\[
	c^*_d(r_d) \leq \cfrac{2\langle\, r_d\, \rangle_a}{\lambda_0}=2\langle\, r_d\, \rangle_a \sqrt{\cfrac{\langle\, d\, \rangle_h}{\langle\, r\, \rangle_a}}=2\sqrt{\langle\, d\, \rangle_h\langle\, r_d\, \rangle_a},
	\]
	which completes the proof of $(1) \Leftrightarrow (2)$.
	
	It remains to prove that $(1) \Leftrightarrow (3)$.
	We first assume the statement $(1)$.
	As in the proof of Theorem \ref{main3}, we obtain $(3)$ and $\lambda
	_0$ only attains the minimum.
	See (\ref{minmax}).
	We next assume the statement $(3)$.
	Since the constant function $\varphi_0$ is the minimizer for $\varphi \mapsto I(\varphi;\lambda_0,d,r)$, we can substituting the constant function into the Euler-Lagrange equation.
	As in the proof of Theorem \ref{main3}, we obtain
	\[
	\cfrac{r}{\langle\, r\, \rangle_a} + \cfrac{\langle\, d\, \rangle_h}{d} = 2,
	\]
	which implies that $(1)$ establishes.
	Finally, we prove that $\varphi_0$ only attains 
	\[
	k_{\lambda_0}(d,r)=\min_{\varphi \in E_L} I(\varphi;\lambda_0,d,r).
	\]
	Assume that $(1)$ holds.
	Since $(2)$ and $(3)$ hold, we have $r=r_d$ and $k_{\lambda_0}(d,r) = I(1/\sqrt{L};\lambda_0,d,r)=-2\langle\, r\, \rangle_a$.
	Take any $\varphi \in E_L$ minimizing $\varphi \mapsto I(\varphi;\lambda_0,d,r)$.
	As in the proof of Theorem \ref{main3}, we obtain
	\begin{eqnarray*}
	-2\langle\, r\, \rangle_a &=& k_{\lambda_0}(d,r)\\
	&=& \int_{0}^{L}d|\varphi '|^2 \ dx -\int_{0}^{L} r \varphi^2 \ dx - \cfrac{\lambda_0^2L^2}{\displaystyle{\int_{0}^{L} \cfrac{1}{d\varphi^2}\  dx}}\\
	&\geq& -\int_{0}^{L} r \varphi^2 \ dx - \cfrac{\lambda_0^2L^2}{\displaystyle{\int_{0}^{L} \cfrac{1}{d\varphi^2}\  dx}}\\
	&\geq& -2\langle\, r\, \rangle_a.
	\end{eqnarray*}
	It implies that
	\[
	\int_{0}^{L}d|\varphi '|^2 \ dx =0.
	\]
	This clearly forces $\varphi \equiv 1/\sqrt{L}$, which is our claim.
\end{proof}

\begin{proof}[Proof of Theorem \ref{main4}]
	Set $d \in C_{\mathrm{per}}(\mathbb{R})$ with $\inf d >0$ and $k>0$.
	It suffices to show that $r_{kd}(x) = r_d(x)$ for any $x \in \mathbb{R}$.
	Since $\langle\, kd\, \rangle_h =k \langle\, d\, \rangle_h$, we obtain
	\begin{eqnarray*}
	r_{kd}(x) &=& \alpha \Big(2-\frac{\langle\, kd\, \rangle_h}{kd(x)}\Big)\\
	 &=& \alpha \Big(2-\frac{k\langle\, d\, \rangle_h}{kd(x)}\Big)\\
	 &=& \alpha \Big(2-\frac{\langle\, d\, \rangle_h}{d(x)}\Big)\\
	 &=& r_d(x) 
	\end{eqnarray*}
	for any $x \in \mathbb{R}$.
	This completes the proof.
\end{proof}

\begin{proof}[Proof of Theorem \ref{main4}]
	We first prove that $(2) \Rightarrow (1)$.
	In the case where $d(x)$ is a constant, by Theorem \ref{main3}, we see that $r(x)$ is a constant.
	It is known that $\psi_c$ is the principal eigenfunction if $d(x)$ and $r(x)$ are constants.
	
	We next prove that $(1) \Rightarrow (2)$.
	We assume that the constant function $\psi_c$ is the principal eigenvalue of the operator $-\mathcal{L}_{\lambda_0,d,r}$.
	By the assumption that $c^*_d(r) = 2\sqrt{\langle\, d\, \rangle_h \langle\, r\, \rangle_a}$, we have $k_{\lambda_0}(d,r)=-2\langle\, r\, \rangle_a$.
	The constant function $\psi_c$ satisfies
	\[
	- (d(x) \psi_c'(x))' - 2\lambda_0 d(x) \psi_c'(x) -(\lambda_0^2 d(x) + \lambda_0 d'(x) +r(x)) \psi_c(x)= k_{\lambda_0}(d,r)\psi_c(x).
	\]
	Thus we have
	\[
	\lambda_0^2 d(x) + \lambda_0 d'(x) +r(x)= 2\langle\, r\, \rangle_a.
	\]
	Dividing this equation by $L$ and integrating it from $0$ to $L$, we get
	\[
	\lambda_0^2 \langle\, d\, \rangle_a +\langle\, r\, \rangle_a= 2\langle\, r\, \rangle_a.
	\]
	Substituting $\lambda_0=\sqrt{\langle\, r\, \rangle_a/\langle\, d\, \rangle_h}$ into this equation, we obtain
	\begin{equation}\label{eha}
	\langle\, d\, \rangle_a = \langle\, d\, \rangle_h.
	\end{equation}
	In general, by the Cauchy-Schwarz inequality, we have
	\begin{equation}\label{ha}
	\Big( \int_{0}^{L} d(x) \,dx\Big) \Big( \int_{0}^{L}\cfrac{1}{d(x)}  \,dx\Big) \geq \Big( \int_{0}^{L} 1 \,dx\Big)^2=L^2,
	\end{equation}
	and this equality holds if and only if $d(x)$ is a constant function.
	From (\ref{ha}), we see that
	\[
	\langle\, d\, \rangle_a \geq \langle\, d\, \rangle_h.
	\]
	The equation (\ref{eha}) means that the equality in (\ref{ha}) holds, which gives $d(x)$ is a constant function.
\end{proof}

\noindent 
\textbf{Acknowledgement.} 
The author would like to thank Professor Hiroshi Matano and Professor Hirokazu Ninomiya for their suggestions and continued encouragement.



\small

\begin{thebibliography}{99}
	
	\bibitem{BH}
	H.~Berestycki and F.~Hamel, 
	\textit{Front propagation in periodic excitable media},
	Communications on Pure and Applied Mathematics, LV, 949-1032, 2002.
	
	\bibitem{BHN1}
	H.~Berestycki, F.~Hamel and N.~Nadirashvili, 
	\textit{The speed of propagation for KPP type problems. I-Periodic framework},
	J. Eur. Math. Soc., 7: 173-213, 2005.
	
	\bibitem{BHR1}
	H.~Berestycki, F.~Hamel and L.~Roques, 
	\textit{Analysis of the periodically fragmented environment model:  I-Species persistence},
	J. Math. Biol. 51: 75-113, 2005.
	
	\bibitem{BHR2}
	H.~Berestycki, F.~Hamel and L.~Roques, 
	\textit{Analysis of the periodically fragmented environment model:  I\hspace{-0.7em}I -Biological invasions and pulsating traveling fronts},
	J. Math. Pures Appl. 84: 1101-1146, 2005.
	
	
	
	\bibitem{condition}
	M.~ElSmaily, F.~Hamel and L.~Roques, 
	\textit{Homogenization and influence of fragmentation in a biological invasion model},
	Disc. Cont. Dyn. Syst. A, 25: 321-342, 2009.
	
	\bibitem{Fisher}
	R.A.~Fisher, 
	\textit{The wave of advance of advantageous genes},
	Ann. Eugenics 7, 335?369, 1937
	
	\bibitem{I}
	R.~Ito, 
	\textit{Analysis of the minimal traveling wave speed via the methods of Young measures},
	SIAM J. Math. Anal., 50 (2018), pp. 3478-3534.
	
	
	\bibitem{S3}
	N.~Kinezaki, K.~Kawasaki, F.~Takasu and N.~Shigesada, 
	\textit{Spatial dynamics of invasion in sinusoidally varying environments},
	Popul. Ecol., 48: 263-270, 2010.
	
	\bibitem{KPP}
	A.~Kolmogorov, I.~Petrovsky, N.~Piskunov, 
	\textit{Etude de lf\'{e}quation de la diffusion
		avec croissance de la quantit\'{e} de mati\`{e}re et son application \`{a} un probl\`{e}me biologique},
	Bjul. Moskowskogo Gos. Univ. Ser. Internat. Sec. A1: 1-26, 1937.
	
	\bibitem{M1}
	X.~Liang, X.~Lin and H.~Matano,
	\textit{A variational problem associated with the minimal speed of travelling wave for spatially periodic reaction-diffusion equation}, 
	Trans. Amer. Math. Soc., 362: 5605-5633, 2010.
	
	\bibitem{M2}
	X.~Liang and H.~Matano, 
	\textit{Maximizing the spreading speed of KPP fronts in two dimensional stratified media}, 
	Pro. London Math. Soc., 109: 1137-1174, 2014.
	
	\bibitem{Xiao}
	R.~Mori, D.~Xiao, 
	\textit{A variational problem associated with the minimal speed of traveling waves for spatially periodic KPP type equations},
	Proc. Lond. Math. Soc. 119, 654-680, 2019.
	
	\bibitem{N1}
	N.~Nadin,
	\textit{The effect of the Schwarz rearrangement on the periodic principal eigenvalue of a nonsymmetric operator},
	SIAM J. Math. Anal., 41:2388-2406, 2010. 
	
	\bibitem{N2}
	N.~Nadin,
	\textit{Some dependence results between the spreading speed and the coefficients of the space-time periodic Fisher-KPP equation},
	European J. Appl. Math., 22:169-185, 2011. 
	
	\bibitem{S1}
	N.~Shigesada, K.~Kawasaki and E.~Teramoto, 
	\textit{Traveling periodic waves in heterogeneous environments},
	Theor. Population Biol., 30: 143-160, 1986.
	
	\bibitem{S}
	J.~G.~Skellam 
	\textit{Random dispersal in theoretical populations},
	Biometrika. 38: 196-218, 1951.
	
	\bibitem{W}
	H.~F.~Weinberger, 
	\textit{On spreading speeds and traveling waves for growth and migration models in a periodic habitat},
	J. Math. Biol. 45: 511-548, 2002.
	

\end{thebibliography}
\end{document}